\documentclass[11pt]{amsart}

\usepackage[T1]{fontenc}
\usepackage[utf8]{inputenc}
\usepackage{lmodern}
\usepackage{amsmath,amssymb,amsthm,mathtools}
\usepackage{booktabs}
\usepackage{enumitem}
\usepackage{graphicx}
\usepackage{microtype}
\usepackage{xcolor}
\usepackage[colorlinks=true,linkcolor=blue,citecolor=blue,urlcolor=blue]{hyperref}
\usepackage{orcidlink}
\usepackage[nameinlink,capitalize]{cleveref}

\setlist[itemize]{leftmargin=2em}
\setlist[enumerate]{leftmargin=2em}

\newtheorem{theorem}{Theorem}
\newtheorem{proposition}{Proposition}
\newtheorem{lemma}{Lemma}

\theoremstyle{definition}

\newtheorem{protocol}{Protocol}

\theoremstyle{remark}

\DeclareMathOperator{\E}{\mathbb E}

\DeclareMathOperator{\SE}{SE}

\newcommand{\Bn}{B_n}
\newcommand{\Om}{\Omega}

\newcommand{\calE}{\mathcal E}

\newcommand{\wideM}{\widehat M}
\newcommand{\widea}{\widehat a}
\newcommand{\wideA}{\widehat A}
\newcommand{\wideR}{\widehat R}

\newcommand{\dedekindaffiliations}{%
  \begin{minipage}{0.94\textwidth}
  \centering
  \textsuperscript{1}Tsung-Dao Lee Institute, Shanghai Jiao-Tong University, Shanghai 201210, Mainland China\\
  \textsuperscript{2}Department of Physics and Institute for Quantum Science
  and Technology, Shanghai University, Shanghai 200444, Mainland China\\
  \textsuperscript{3}School of Economics, Shanghai University, Shanghai 200444, Mainland China\\
  \textsuperscript{4}Department of Physics, Tamkang University, New Taipei 251, Taiwan\\
  \textsuperscript{5}Physics Division, National Center for Theoretical Sciences, Taipei 10617, Taiwan\\
  \textsuperscript{6}Shanghai Key Laboratory of High Temperature
  Superconductors, Shanghai University, Shanghai 200444, Mainland China
  \par\vspace{0.45em}
  E-mails:
  \href{mailto:cts2003912@shu.edu.cn}{cts2003912@shu.edu.cn},
  \href{mailto:fenghaozi@shu.edu.cn}{fenghaozi@shu.edu.cn},
  \href{mailto:wanghaozhe@shu.edu.cn}{wanghaozhe@shu.edu.cn},
  \href{mailto:chianshu@gmail.com}{chianshu@gmail.com}
  \href{mailto:kilar@shu.edu.cn}{kilar@shu.edu.cn}
  \end{minipage}%
}

\newcommand{\authororcid}[1]{%
  \texorpdfstring{\textsuperscript{\orcidlink{#1}}}{}%
}

\makeatletter
\def\@setauthors{%
  \begingroup
  \def\thanks{\protect\thanks@warning}%
  \trivlist
  \centering\normalsize \@topsep30\p@\relax
  \advance\@topsep by -\baselineskip
  \item\relax
  \author@andify\authors
  \def\\{\protect\linebreak}%
  \authors\par
  \vskip.7em
  {\scriptsize\dedekindaffiliations\par}%
  \endtrivlist
  \endgroup
}
\def\@setaddresses{}
\def\subsection{\@startsection{subsection}{2}%
  \z@{.5\linespacing\@plus.7\linespacing}{.35\linespacing}%
  {\normalfont\bfseries}}
\makeatother

\title[Layer-Ratio Estimate of Dedekind Numbers]{Finite-\(n\) Estimate of Dedekind Numbers by Layer-Ratio Monte Carlo}

\author[T.-S. Chen]{Tian-Shun Chen\authororcid{0009-0004-4744-4597}\textsuperscript{1,2}}

\author[H. Feng]{Hao Feng\authororcid{0009-0003-3388-1094}\textsuperscript{2}}

\author[H. Wang]{Haozhe Wang\authororcid{0009-0009-6252-5543}\textsuperscript{3}}

\author[C.-S. Chen]{Chian-Shu Chen\authororcid{0000-0002-3354-0105}\textsuperscript{4,5,*}}

\author[K. Zhang]{Kilar Zhang\authororcid{0000-0001-6776-9074}\textsuperscript{2,6,*}}
\thanks{*Corresponding authors.}

\begin{document}

\begin{abstract}
Dedekind's problem counts monotone Boolean functions, equivalently downsets of a Boolean lattice.  We recast this enumeration as a finite layer-ratio reconstruction problem for the Whitney numbers of the ranked ideal lattice.  An exact adjacent-layer double count expresses each layer ratio through local averages of the number of addable elements and the number of removable elements. Reversible fixed-layer Markov chains estimate these averages and hence estimate the Dedekind number \(M(n)\).  Backtests at \(M(8)\) and \(M(9)\) calibrate seed-level variability under the fixed protocol and measure the observed Monte Carlo budget scaling.  The resulting estimate probes the Whitney-number sequence of the ideal lattice.  Although these rows have previously been described empirically as unimodal, the high-precision \(n=9\) estimate has a shallow two-shoulder feature around the central rank, contrary to that empirical description; \(n=11\) and \(n=13\) center-window estimates show a larger-contrast analogous pattern.  The protocol estimate for \(M(10)\) is
\[
\widehat M(10)=(8.9360\pm0.0010)\times 10^{78},
\]
where the displayed uncertainty is the budget-based forecast scale from the cross-\(n\) scaling law under the production budget.
\end{abstract}

\maketitle

\clearpage
\tableofcontents

\section{Introduction}

Dedekind's problem asks for the number \(M(n)\) of monotone Boolean
functions on \(n\) variables.  Equivalently, \(M(n)\) equals each of the
following quantities: the number of antichains in the Boolean lattice
\(\Bn=\{0,1\}^n\) under the coordinatewise order; the number of downsets of
\(\Bn\); and the cardinality of the free distributive lattice on \(n\)
generators.  The problem goes back to Dedekind's 1897 work on free
distributive structures \cite{Dedekind1897}; the free-distributive-lattice
and early numerical literature attributes to Church, Ward, and Yamamoto
\cite{Church1940,Ward1946,Yamamoto1954}.  Its elementary formulation hides an
extreme computational difficulty: the known exact values currently stop at
\(n=9\).

The values up to \(n=7\) were obtained through a sequence of increasingly
refined enumerations
\cite{Church1940,Ward1946,Yamamoto1954}; Wiedemann computed \(M(8)\) in 1991
\cite{Wiedemann1991}, and later algorithms of Fidytek et al.\ gave an independent algorithmic
confirmation \cite{FidytekEtAl2001}.  Subsequent work developed recursive, interval, and
downset-enumeration approaches to the same antichain lattice
\cite{DeCausmaeckerDeWannemacker2014,BermanKoehler2021,Campo2022}.  A related
symmetry line counts inequivalent monotone Boolean functions and fixed points
of variable permutations, using Burnside-type reductions
\cite{StephenYusun2014,Pawelski2022,Szepietowski2022,Pawelski2024}.  More
than three decades after \(M(8)\) was obtained, \(M(9)\) was computed by two independent
projects.  Jaekel used a matrix formulation together with symmetries of the
free distributive lattice and formal concept analysis \cite{Jaekel2023}.  Van
Hirtum, De Causmaecker and collaborators used a P-coefficient formula,
equivalence-class reductions, and FPGA supercomputing
\cite{VanHirtumEtAl2023FPGA}; the mathematical form of this computation and
its extensions are developed further in \cite{DeCausmaeckerVanHirtum2024}.
Independent congruence checks for the ninth Dedekind number were also obtained
by Pawelski and Szepietowski \cite{PawelskiSzepietowski2023}.
Both computations produced
\[
  M(9)=286386577668298411128469151667598498812366.
\]
These exact computations reduce an enormous finite sum by algebraic structure,
symmetry, interval decompositions, and specialized hardware.  Their primary
output is the total \(M(n)\).  This should be distinguished from the finer
Whitney-number sequence
\[
  a_n(k)=\#\{D\subseteq \Bn:\ D\text{ is a downset and } |D|=k\},
  \qquad 0\le k\le 2^n,
\]
where \(a_n(k)\) is the \(k\)-th Whitney number of the second kind
(or rank number) of \(\mathcal I(\Bn)\), ranked by ideal cardinality
\cite[Ch.~3]{Stanley1986}.  This sequence records how the total Dedekind number
is distributed over the ranks of the ideal lattice.  To our knowledge, exact
complete rows of this Whitney-number sequence are currently available only
through \(n=7\) \cite{OEISA269699}.  One advantage of the layer-ratio approach
below is that it reconstructs this sequence directly, rather than merely
estimating its sum.

In parallel, a separate line of work has studied Dedekind's problem
asymptotically.  Kleitman proved that \(\log_2 M(n)\) is asymptotic to the
size of a largest layer of the Boolean lattice \cite{Kleitman1969}; the error
term was sharpened by Kleitman and Markowsky \cite{KleitmanMarkowsky1975}.
Korshunov later obtained asymptotics for \(M(n)\) itself
\cite{Korshunov1977,Korshunov2003}.  Kahn gave an entropy-based proof of the
Kleitman-Markowsky bound through independent sets and antichains
\cite{Kahn2002}; the same independent-set viewpoint also appears in related
work on maximal antichains \cite{IlincaKahn2013}.  Korshunov and Shmulevich
studied the distribution of monotone Boolean functions by the number of lower
units, equivalently the number of terms in the minimal DNF
\cite{KorshunovShmulevich2002}.  More recently, cluster-expansion methods
from statistical physics have yielded refined asymptotics for Dedekind's
problem and for antichains of prescribed size
\cite{JenssenMalekshahianPark2024}.  
These results explain why most of the
mass is controlled by the central layers of the Boolean lattice, and they give
powerful asymptotic information. Recent work has also developed related
variants and generalizations of Dedekind-type counting problems
\cite{BiswasSarkar2025,PawelskiSzepietowski2025,ParkSarantisTetali2025,
FalgasRavryRatyTomon2026,JenssenParkSarantis2026}. Our objective is to
construct a finite-\(n\) numerical estimator for values such as \(M(10)\).

We develop a finite-\(n\) sampling method that reconstructs the Whitney numbers
of the downset lattice.
Decompose the set of all downsets by cardinality:
\[
  \Om_{n,k}
  =
  \{D\subseteq \Bn: D\text{ is a downset and } |D|=k\},
  \qquad
  a_n(k)=|\Om_{n,k}|.
\]
Then
\[
  M(n)=\sum_{k=0}^{2^n} a_n(k).
\]
Thus \(M(n)\) is the sum of the finite Whitney-number sequence
\((a_n(k))_{k=0}^{2^n}\).

For a downset \(D\in\Om_{n,k}\), let
\(A(D)\) and \(R(D)\) be the numbers of elements that can be added to,
respectively removed from, \(D\) while remaining a downset.  Double-counting
cover edges between adjacent layers gives
\[
  a_n(k)\,\E_k A
  =
  a_n(k+1)\,\E_{k+1}R,
  \qquad
  \frac{a_n(k+1)}{a_n(k)}
  =
  \frac{\E_k A}{\E_{k+1}R},
\]
where \(\E_k\) denotes expectation under the uniform distribution on
\(\Om_{n,k}\).  Thus the global sequence of Whitney numbers can be
reconstructed from fixed-layer averages.

Conceptually, the ratio step is in the spirit of broad-histogram identities,
but here it is specialized to the fixed-cardinality layers of the downset
lattice.  General Monte Carlo schemes for finite-set size estimation, such as
cascading exclusion \cite{ChatterjeeDiaconisHolmes2026}, follow a different
statistical route; the present method uses the adjacent-layer geometry
specific to downsets of the Boolean lattice.  The main contributions are as
follows.

\begin{itemize}
  \item We formulate Dedekind-number estimation as a finite-\(n\)
        reconstruction problem for the Whitney numbers
        \((a_n(k))_{k=0}^{2^n}\) of the downset lattice, rather than only for
        the total count \(M(n)\).

  \item We give a fixed-layer Monte Carlo implementation of the adjacent-ratio
        estimator, together with a log-space reconstruction procedure, validate
        the resulting estimates at the known values \(M(8)\) and \(M(9)\), and
        use those cases to calibrate seed-level variability, where a seed means
        one independent repetition of the fixed protocol.

  \item We apply the calibrated protocol to estimate \(M(10)\) and analyze
        the finite-\(n\) layer shape, including the two-shoulder structure at
        \(n=9\) and the higher-contrast odd-dimensional center-window patterns
        at \(n=11\) and \(n=13\).
\end{itemize}

The remainder of the paper is organized as follows.
\Cref{sec:theory-method} develops the finite-\(n\) layer-ratio framework,
including the layer decomposition, the adjacent-layer identity, fixed-layer
sampling, deterministic reconstruction, consistency, and the numerical
protocol.  \Cref{sec:results} presents the main numerical results, including
the known-value validation, the \(M(10)\) estimate, and the reconstructed
Whitney-number shapes.  Appendix~\ref{sec:numerical-validation-details}
records the \(M(10)\) production protocol and additional validation
diagnostics.

\section{Theory and Method}
\label{sec:theory-method}

\subsection{Layer Decomposition and Ratio Identity}
\label{sec:layer-decomposition}

We use the coordinate model of the Boolean lattice:
\(\Bn=\{0,1\}^n\).  Its order is the coordinatewise order.  Thus, for
\(x,y\in\Bn\),
\[
  x\preceq y
  \quad\Longleftrightarrow\quad
  x_i\le y_i\text{ for every }1\le i\le n .
\]
We write \(x\prec y\) when \(x\preceq y\) and \(x\ne y\).  

A subset \(D\subseteq \Bn\) is a downset if
\[
  x\in D,\ y\preceq x \quad\Longrightarrow\quad y\in D .
\]
The \(n\)th Dedekind number is
\[
  M(n)=\#\{D\subseteq \Bn: D\text{ is a downset}\}.
\]

Equivalently, \(M(n)\) counts monotone Boolean functions and antichains.  Let
\(N=|\Bn|=2^n\).  The downsets are ranked by cardinality: for
\(0\le k\le N\), write
\[
  \Om_{n,k}=\{D\subseteq \Bn: D\text{ is a downset and } |D|=k\}
\]
and
\[
  a_n(k)=|\Om_{n,k}|.
\]
The number \(a_n(k)\) is the \(k\)-th Whitney number of the second kind
(or rank number) of the ranked ideal lattice \(\mathcal I(\Bn)\).
The sequence \((a_n(k))\) is therefore the
Whitney-number sequence of \(\mathcal I(\Bn)\), and
\[
  M(n)=\sum_{k=0}^{2^n} a_n(k).
\]
When \(n\) is fixed, we suppress it from the notation when this causes no
ambiguity.  Let \(\mu_{n,k}\) be the uniform probability measure on
\(\Om_{n,k}\).  For any real-valued function
\(F:\Om_{n,k}\to\mathbb R\), define
\[
  \E_k F
  =
  \frac{1}{a_n(k)}\sum_{D\in\Om_{n,k}}F(D).
\]

There is a useful boxed-partition viewpoint for these definitions.  An
\(n\)-dimensional partition may be viewed as a finite set of boxes in
\(\mathbb Z_{\ge0}^n\) satisfying the melting rule: whenever a box is present,
all coordinate-wise smaller boxes are present.  This is the box form of
MacMahon's plane partition and higher-dimensional partitions
\cite{MacMahon1898PartitionsII,MacMahon1912PartitionsVI}, using the recent
terminology of \cite{XiangFengZhuoChenZhang2026}.
Restricting the ambient corner to the \(2\times\cdots\times2\) box
\(\{0,1\}^n\) gives exactly the downsets of \(\Bn\).  Thus \(M(n)\) is the
number of legal boxed configurations in the \(n\)-dimensional Boolean box of
side length \(2\), and the Whitney number \(a_n(k)\) counts those
configurations with exactly \(k\) occupied boxes; see
\cref{fig:boxed-boolean-layers}.

\begin{figure}[t]
\centering
\includegraphics[width=\textwidth,height=0.74\textheight,keepaspectratio]{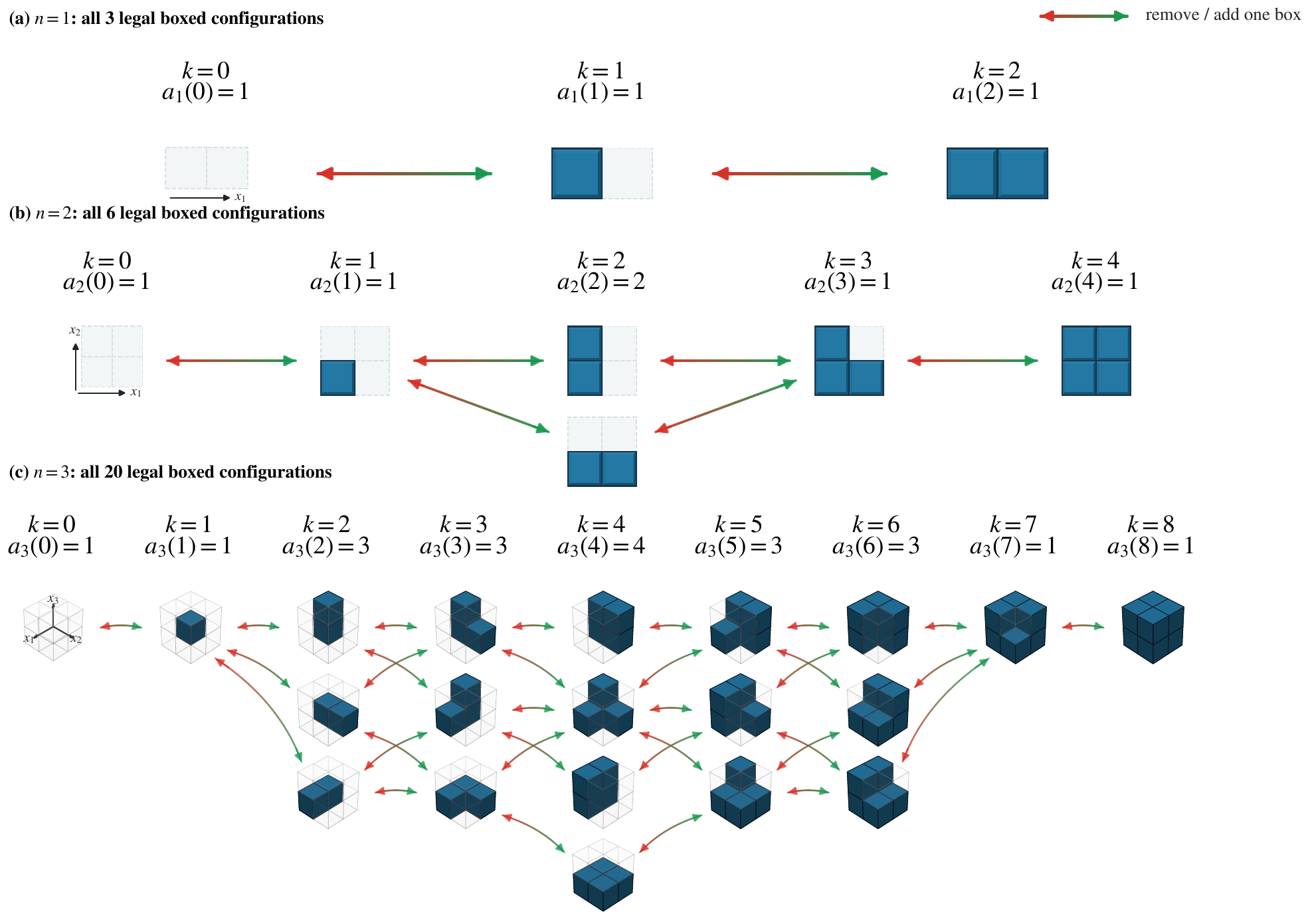}
\caption{Boxed-partition view of the first Boolean boxes.  Colored boxes form a
downset: if a box is occupied, every coordinate-wise smaller box is occupied.
Within each dimension, configurations are grouped by their number \(k\) of
occupied boxes, i.e. by the cardinality layer \(\Omega_{n,k}\).  The Whitney
numbers shown in each row sum to \(M(1)=3\), \(M(2)=6\), and \(M(3)=20\).
Green arrows indicate adjacent-layer moves that add one box; red arrows indicate
the reverse moves that remove one box.}
\label{fig:boxed-boolean-layers}
\end{figure}

For \(D\in\Om_{n,k}\), let
\[
  A(D)=\#\{x\in \Bn\setminus D: D\cup\{x\}\text{ is a downset}\}
\]
and
\[
  R(D)=\#\{x\in D: D\setminus\{x\}\text{ is a downset}\}.
\]
Equivalently,
\[
  A(D)=|\min(\Bn\setminus D)|,
  \qquad
  R(D)=|\max(D)|,
\]
where the minimum and maximum are taken with respect to \(\preceq\).

The global enumeration can now be expressed in terms of cover edges between
adjacent cardinality layers.  Let
\[
  \calE_{n,k}
  =
  \{(D,\Gamma)\in\Om_{n,k}\times\Om_{n,k+1}:D\subseteq\Gamma\}
\]
be the set of cover edges between the two adjacent layers.  Since
\(|\Gamma|=|D|+1\), every edge has the form \(\Gamma=D\cup\{x\}\) for a unique
element \(x\in\Bn\setminus D\).  Equivalently, \(x\) is addable for \(D\), and
the same \(x\) is removable for \(\Gamma\).

\begin{theorem}[Layer-ratio reconstruction]
\label{thm:layer-ratio-reconstruction}
For \(0\le k<N=2^n\),
\[
  a_n(k)\,\E_k A = a_n(k+1)\,\E_{k+1} R .
\]
Consequently,
\[
  \frac{a_n(k+1)}{a_n(k)}
  =
  \frac{\E_k A}{\E_{k+1} R}.
\]
If the addable/removable averages \(\E_k A\) and \(\E_k R\) are known for all relevant
layers, then the whole Whitney-number sequence is determined by
\[
  a_n(0)=1,\qquad
  a_n(k+1)
  =
  a_n(k)\frac{\E_k A}{\E_{k+1}R}
  \quad (0\le k<N).
\]
Equivalently,
\[
  a_n(k)
  =
  \prod_{j=0}^{k-1}\frac{\E_j A}{\E_{j+1}R},
  \qquad
  M(n)=\sum_{k=0}^{N}a_n(k).
\]
\end{theorem}

\begin{proof}
We count the same finite edge set \(\calE_{n,k}\) in two ways.  From the lower
layer, the number of edges incident to a downset \(D\in\Om_{n,k}\) is exactly
\(A(D)\).  Therefore
\[
  |\calE_{n,k}|
  =
  \sum_{D\in\Om_{n,k}} A(D)
  =
  a_n(k)\,\E_k A.
\]
From the upper layer, the number of edges incident to
\(\Gamma\in\Om_{n,k+1}\) is exactly \(R(\Gamma)\).  Hence
\[
  |\calE_{n,k}|
  =
  \sum_{\Gamma\in\Om_{n,k+1}} R(\Gamma)
  =
  a_n(k+1)\,\E_{k+1}R.
\]
The two expressions count the same edge set, so they are equal.  Since
\(a_n(k)>0\) and \(\E_{k+1}R>0\), division gives the ratio identity.
Iterating this identity from the endpoint value \(a_n(0)=1\)
yields the reconstruction formulas, and summing the reconstructed Whitney
numbers gives \(M(n)\).
\end{proof}

The adjacent-layer identity is a finite-poset analogue of a broad-histogram
relation: consecutive Whitney-number ratios are obtained from fixed-layer
averages of addable and removable element counts, without imposing a
parametric model on the sequence \(a_n(k)\)
\cite{DeOliveiraPennaHerrmann1996}.

\subsection{Fixed-Layer Sampling}
\label{sec:fixed-layer-sampling}

For a fixed \(n\) and a fixed layer \(0\le k\le N=2^n\), the Markov chain
used in this paper has state space \(\Om_{n,k}\).  Its purpose is to sample
approximately from the uniform layer measure \(\mu_{n,k}\), so that averages of
the addable/removable statistics \(A\) and \(R\) approximate the microcanonical
expectations \(\E_k A\) and \(\E_k R\).  The sampling and boundary-measurement
mechanism is illustrated in the three-dimensional example in
\cref{fig:layer-mcmc-schematic}.

\begin{figure}[t]
\centering
\includegraphics[width=0.88\textwidth]{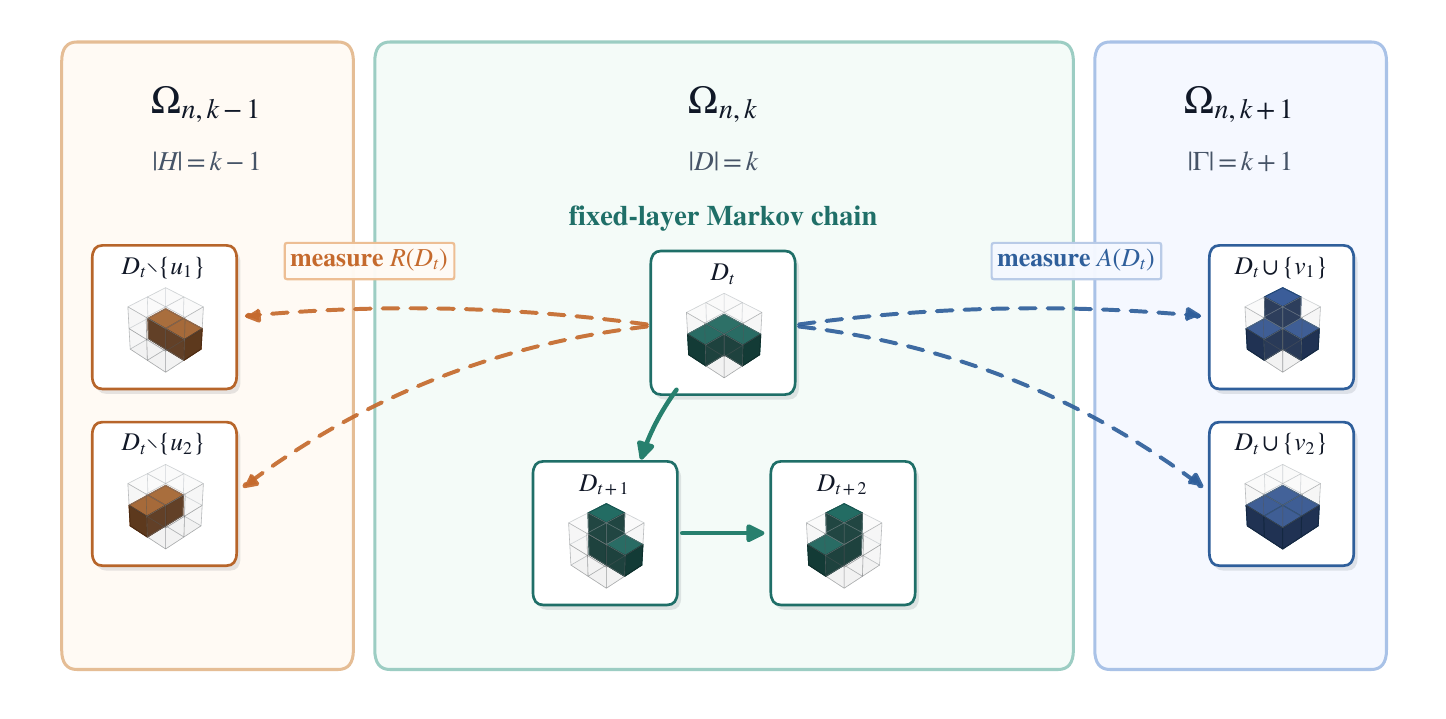}
\caption{Three-dimensional \(B_3\) schematic of the fixed-layer sampling
mechanism.  Solid green arrows inside \(\Omega_{n,k}\) represent exchange
moves of the Markov chain within the fixed-cardinality layer; the miniature
states are valid \(B_3\) downsets drawn for visualization.  Dashed orange
arrows to \(\Omega_{n,k-1}\) and dashed blue arrows to \(\Omega_{n,k+1}\)
represent boundary measurements of removable and addable choices used to
estimate \(\E_k R\) and \(\E_k A\), not transitions of a global chain.
Adjacent layer ratios are then reconstructed from the exact adjacent-layer
identity.}
\label{fig:layer-mcmc-schematic}
\end{figure}

The elementary move is an exchange move.  Starting from a downset
\(D\in\Om_{n,k}\), first delete a removable vertex and then add an addable
vertex:
\[
  D
  \longrightarrow
  D^- = D\setminus\{u\}
  \longrightarrow
  \Gamma = D^-\cup\{v\}.
\]
Here
\[
  u\in\max(D),
  \qquad
  v\in\min(\Bn\setminus D^-).
\]
By the extremal characterization of addable and removable elements above,
deleting \(u\) preserves the downset property, and adding \(v\) to \(D^-\)
also preserves it.  Therefore
\(\Gamma\in\Om_{n,k}\).  If \(v=u\), then \(\Gamma=D\), giving a natural
self-loop.

The exchange graph \(G_{n,k}\) has vertex set \(\Om_{n,k}\), with an edge
between two distinct states \(D\) and \(\Gamma\) if they differ by deleting one
element and adding one element as above.  The Markov chain is a
Metropolis-Hastings chain on this graph, with the self-loops coming both from
trivial proposals and from rejected nontrivial proposals
\cite{MetropolisEtAl1953,Hastings1970}.

Assume first that \(0<k<N\).  Given \(D\in\Om_{n,k}\), the proposal is:
\begin{enumerate}
  \item choose \(u\) uniformly from the \(R(D)\) removable vertices of \(D\);
  \item set \(D^-=D\setminus\{u\}\);
  \item choose \(v\) uniformly from the \(A(D^-)\) addable vertices of \(D^-\);
  \item propose \(\Gamma=D^-\cup\{v\}\).
\end{enumerate}
For a nontrivial proposal \(D\ne\Gamma\), the removed and added vertices are
unique.  Write
\[
  D\setminus\Gamma=\{u\},
  \qquad
  \Gamma\setminus D=\{v\},
  \qquad
  H=D\setminus\{u\}=\Gamma\setminus\{v\}.
\]
Then the proposal probability from \(D\) to \(\Gamma\) is
\[
  q(D,\Gamma)
  =
  \frac{1}{R(D)A(H)}.
\]
The reverse proposal deletes \(v\) from \(\Gamma\) and adds \(u\) back to the
same intermediate downset \(H\), so
\[
  q(\Gamma,D)
  =
  \frac{1}{R(\Gamma)A(H)}.
\]

The target distribution on \(\Om_{n,k}\) is uniform.  Hence the
Metropolis-Hastings acceptance probability for a nontrivial proposal is
\[
  \alpha(D,\Gamma)
  =
  \min\left\{1,\frac{q(\Gamma,D)}{q(D,\Gamma)}\right\}
  =
  \min\left\{1,\frac{R(D)}{R(\Gamma)}\right\}.
\]
If the proposal is rejected, the chain remains at \(D\).  If \(v=u\), the
proposal is already \(D\) and is treated as an accepted self-loop.  The endpoint
layers \(k=0\) and \(k=N\) are singletons, so the chain is the trivial stationary
chain there.

\begin{lemma}[Stationarity]
\label{lem:stationarity}
For each \(0\le k\le N\), the uniform distribution \(\mu_{n,k}\) on
\(\Om_{n,k}\) is stationary for the transition kernel described above.
\end{lemma}

\begin{proof}
For \(k=0\) and \(k=N\), the state space is a singleton, so the claim is
immediate.  Suppose \(0<k<N\).  It is enough to verify detailed balance for
distinct neighboring states \(D,\Gamma\in\Om_{n,k}\).  Since \(\mu_{n,k}\) is
uniform, detailed balance reduces to
\[
  q(D,\Gamma)\alpha(D,\Gamma)
  =
  q(\Gamma,D)\alpha(\Gamma,D).
\]
Using the formula above, with \(H=D\cap\Gamma\), the left-hand side is
\[
  \frac{1}{R(D)A(H)}
  \min\left\{1,\frac{R(D)}{R(\Gamma)}\right\}
  =
  \frac{1}{A(H)\max\{R(D),R(\Gamma)\}}.
\]
The same expression is obtained after exchanging \(D\) and \(\Gamma\).  Hence
detailed balance holds for every off-diagonal transition.  The diagonal terms
then balance automatically because each row of the transition matrix sums to
one.  Therefore \(\mu_{n,k}\) is stationary.
\end{proof}

\begin{lemma}[Connectivity of fixed layers]
\label{lem:layer-connectivity}
For every \(0\le k\le N\), the exchange graph \(G_{n,k}\) is connected.
\end{lemma}

\begin{proof}
The cases \(k=0\) and \(k=N\) are trivial.  Let \(0<k<N\), and take two
states \(D,E\in\Om_{n,k}\).  If \(D=E\), there is nothing to prove.

Assume \(D\ne E\).  Choose an element \(u\) maximal in \(D\setminus E\) with
respect to \(\preceq\).  Then \(u\) is also maximal in \(D\).  Indeed, suppose
that there exists \(z\in D\) strictly above \(u\), that is, \(u\prec z\).  If
\(z\notin E\), then \(z\in D\setminus E\), contradicting
the maximality of \(u\) in \(D\setminus E\).  If \(z\in E\), then the downset
property of \(E\), together with \(u\prec z\), implies \(u\in E\), again a
contradiction.  Hence \(u\in\max(D)\), so deleting \(u\) preserves the downset
property.

Next choose an element \(v\) minimal in \(E\setminus D\), again with respect to
\(\preceq\).  We claim that \(v\) is addable to
\(D^-:=D\setminus\{u\}\).  Let \(y\prec v\).  Since \(v\in E\) and \(E\)
is a downset, we have \(y\in E\).  If \(y\notin D\), then \(y\in E\setminus D\),
contradicting the minimality of \(v\).  Thus every \(y\prec v\)
lies in \(D\).  Moreover \(y\ne u\): if \(y=u\), then \(u\prec v\) and
\(v\in E\) would imply \(u\in E\), contradicting \(u\in D\setminus E\).
Therefore every \(y\prec v\) lies in \(D^-\), so
\(D^-\cup\{v\}\) is a downset.

Thus
\[
  D'=(D\setminus\{u\})\cup\{v\}
\]
is a legal exchange move in \(G_{n,k}\).  Since \(u\in D\setminus E\) is
removed and \(v\in E\setminus D\) is added, we have
\[
  |D'\setminus E|=|D\setminus E|-1 .
\]
Repeating the same argument finitely many times reaches \(E\).  Hence every
two states in \(\Om_{n,k}\) are connected by legal exchange moves, so
\(G_{n,k}\) is connected.
\end{proof}

\begin{proposition}[Ergodicity of the fixed-layer chain]
\label{prop:ergodic}
For every \(0\le k\le N\), the fixed-layer chain is irreducible and aperiodic.
Consequently, for fixed \(n\) and \(k\), the empirical averages of \(A\) and
\(R\) along the chain, from any initial state, converge almost surely to
\(\E_k A\) and \(\E_k R\), respectively.  Discarding any fixed finite burn-in
does not change these limits.
\end{proposition}

\begin{proof}
Irreducibility follows from \cref{lem:layer-connectivity}, because every edge
of \(G_{n,k}\) has positive proposal probability and positive acceptance
probability.

For aperiodicity, consider first \(0<k<N\).  At any state \(D\), choose any
removable vertex \(u\in\max(D)\).  After deleting \(u\), the same vertex \(u\)
is addable to \(D\setminus\{u\}\).  Thus the proposal can choose \(v=u\), which
returns immediately to \(D\) with positive probability.  Hence every state has
a positive self-loop.  The singleton layers \(k=0\) and \(k=N\) are also
aperiodic.  Therefore the chain is aperiodic in all layers.  Together with
\cref{lem:stationarity}, the finite-state Markov-chain ergodic theorem applies
to any real-valued function on \(\Om_{n,k}\), in particular to \(A\) and \(R\)
\cite[Chapter~4]{LevinPeresWilmer2009}.
\end{proof}

\subsection{Estimator, Reconstruction, and Consistency}
\label{sec:estimator-reconstruction}

Fix \(n\), and let \(N=2^n\).  The estimator works layer by layer.  For a
sampled layer \(k\), let \(C_k\) be the number of fixed-layer chains run on
\(\Om_{n,k}\).  Chain \(c\) contributes \(m_{k,c}\) recorded states after
burn-in and thinning; we write these states as
\[
  D_{k,c,1},\ldots,D_{k,c,m_{k,c}}\in\Om_{n,k}.
\]
The total number of recorded states in layer \(k\) is
\[
  S_k=\sum_{c=1}^{C_k}m_{k,c}.
\]
The empirical layer means of the addable and removable counts are then
\[
  \wideA_k
  =
  \frac{1}{S_k}
  \sum_{c=1}^{C_k}\sum_{t=1}^{m_{k,c}} A(D_{k,c,t}),
  \qquad
  \wideR_k
  =
  \frac{1}{S_k}
  \sum_{c=1}^{C_k}\sum_{t=1}^{m_{k,c}} R(D_{k,c,t}).
\]
Here \(A(D)\) and \(R(D)\) are the addable and removable counts defined in
\cref{sec:layer-decomposition}.
The endpoint means needed for adjacent ratios
are exact:
\[
  \wideA_0=1,\qquad \wideR_N=1,
\]
because the empty downset has exactly one addable element, and the full
downset \(\Bn\) has exactly one removable element.

The Boolean lattice has an order-reversing complement map
\[
  x=(x_1,\ldots,x_n)
  \longmapsto
  x^c=(1-x_1,\ldots,1-x_n).
\]
It induces a duality on downsets:
\[
  \theta(D)
  =
  \{x\in\Bn:x^c\notin D\}.
\]
Equivalently, \(\theta(D)\) is the complement in \(\Bn\) of the image of \(D\)
under \(x\mapsto x^c\).

\begin{lemma}[Layer duality]
\label{lem:layer-duality}
For every downset \(D\subseteq\Bn\), \(\theta(D)\) is a downset,
\(|\theta(D)|=N-|D|\), and \(\theta(\theta(D))=D\).  Consequently
\[
  a_n(k)=a_n(N-k)
  \qquad (0\le k\le N).
\]
Moreover,
\[
  A(D)=R(\theta(D)),
  \qquad
  R(D)=A(\theta(D)).
\]
Hence
\[
  \E_k A=\E_{N-k}R,
  \qquad
  \E_k R=\E_{N-k}A.
\]
\end{lemma}

\begin{proof}
The complement map is order-reversing.  Therefore taking the complement of the
image of \(D\) sends downsets to downsets, changes the size from \(|D|\) to
\(N-|D|\), and is its own inverse.  Under the same order-reversing bijection,
minimal elements of \(\Bn\setminus D\) correspond to maximal elements of
\(\theta(D)\), and maximal elements of \(D\) correspond to minimal elements of
\(\Bn\setminus\theta(D)\).  Hence addable and removable vertices are exchanged,
which gives the identities for \(A\) and \(R\).  Averaging over the bijection
\(\theta:\Om_{n,k}\to\Om_{n,N-k}\) gives the expectation identities.
\end{proof}

The symmetry \(a_n(k)=a_n(N-k)\) is the same rank symmetry recorded for the
level polynomials of free distributive lattices by Markowsky
\cite{Markowsky1980}; with the two endpoint ideals omitted, it appears in OEIS
A269699 as \(T(n,k)=T(n,2^n-k)\) \cite{OEISA269699}.  The same duality map
also exchanges adjacent-layer boundaries: the addable elements of \(D\) are in
bijection with the removable elements of \(\theta(D)\), and conversely. 

The use of \cref{lem:layer-duality} is not a prior estimate of the unknown
answer.  It is an exact automorphism identity of the finite poset.  In
practice it lets us mirror sampled layer summaries, reduce redundant work, and
check whether independently sampled mirror layers agree within their empirical
uncertainty.

For the reported half-row reconstructions, define the exact adjacent ratio and
its plug-in estimator on the sampled side by
\[
  \rho_k=\frac{a_n(k+1)}{a_n(k)},
  \qquad
  \widehat\rho_k=\frac{\wideA_k}{\wideR_{k+1}},
  \qquad 0\le k<N/2.
\]
By \cref{thm:layer-ratio-reconstruction},
\[
  \rho_k=\frac{\E_k A}{\E_{k+1}R},
\]
so \(\widehat\rho_k\) is obtained by replacing the two exact layer averages by
their sampled estimates.

The reconstruction is performed on the logarithmic scale.  Set
\[
  \widehat y_k
  =
  \log\widehat\rho_k
  =
  \log\wideA_k-\log\wideR_{k+1},
  \qquad 0\le k<N/2.
\]
Starting from the exact endpoint value \(a_n(0)=1\), form the cumulative
log Whitney numbers on the sampled half
\[
  x_0=0,\qquad
  x_k=\sum_{j=0}^{k-1}\widehat y_j
  \quad (1\le k\le N/2).
\]
In the reported numerical reconstructions, only one side of each dual pair is
used for the production estimate.  The remaining layers are filled by the
exact rank duality \(a_n(k)=a_n(N-k)\):
\[
  \widehat x_k=
  \begin{cases}
    x_k, & 0\le k\le N/2,\\
    x_{N-k}, & N/2<k\le N.
  \end{cases}
\]
The reconstructed Whitney number is
\[
  \widea_n(k)=\exp(\widehat x_k),
\]
and the Dedekind-number estimate is computed stably by log-sum-exp:
\[
  \log \wideM(n)
  =
  m+\log\sum_{k=0}^{N}\exp(\widehat x_k-m),
  \qquad
  m=\max_{0\le k\le N}\widehat x_k .
\]
Once the addable and removable averages have been sampled, the reported
\(M(10)\) estimate is determined entirely by this reconstruction and the exact
Boolean-lattice duality; the procedure introduces neither fitted smoothing
weights nor penalty parameters.

\begin{theorem}[Fixed-\(n\) consistency]
\label{thm:consistency}
Fix \(n\), and write \(N=2^n\).  For each layer \(0\le k\le N\), suppose that
the layer averages used in the reconstruction satisfy
\[
  \wideA_k\longrightarrow \E_k A,
  \qquad
  \wideR_k\longrightarrow \E_k R
\]
in probability for all needed \(k\).  Let \(\wideM(n)\) be obtained from the
estimated log-ratios by the deterministic log Whitney-number reconstruction described
above.  Then
\[
  \wideM(n)\longrightarrow M(n)
\]
in probability.  If the addable and removable averages converge almost surely,
then \(\wideM(n)\) also converges to \(M(n)\) almost surely.
\end{theorem}

\begin{proof}
All exact means appearing in the ratios are strictly positive.  Therefore,
by the continuous mapping theorem,
\[
  \widehat y_k
  =
  \log\wideA_k-\log\wideR_{k+1}
  \longrightarrow
  \log\E_k A-\log\E_{k+1}R
  =
  \log\rho_k
\]
in probability, where the last equality follows from the adjacent-layer
identity in \cref{thm:layer-ratio-reconstruction}.  Since \(n\) is fixed,
there are only finitely many layers, so the estimated log-ratio vector
converges to the exact log-ratio vector.

The reconstruction from log-ratios to log Whitney numbers is a continuous
finite-dimensional map.  With the exact ratios, anchored at \(a_n(0)=1\), it
recovers the exact Whitney numbers and hence
\(\sum_k a_n(k)=M(n)\).  The continuous mapping theorem therefore gives
\(\wideM(n)\to M(n)\) in probability.  If the fixed-layer sample averages converge almost surely, the
same continuity argument gives almost-sure convergence.
\end{proof}

\begin{proposition}[Fixed-\(n\) Monte Carlo scaling]
\label{prop:fixed-n-clt}
Fix \(n\) and a fixed layer-sampling protocol.  Let \(K\) be the set of
adjacent log-ratios estimated by the protocol.  For \(k\in K\), write
\[
  y_k^\ast=\log\frac{a_n(k+1)}{a_n(k)}
  =
  \log\E_k A-\log\E_{k+1}R,
  \qquad
  \widehat y_k=\log\wideA_k-\log\wideR_{k+1}.
\]
Let \(y^\ast=(y_k^\ast)_{k\in K}\) and
\(\widehat y=(\widehat y_k)_{k\in K}\).  Let \(B\) be the total number of
recorded states and assume that finite-chain initialization bias is negligible
and that the sampled layer budgets grow with fixed positive proportions.  Then
the finite-state Markov-chain CLT and the delta method give
\[
  \sqrt B\bigl(\widehat y-y^\ast\bigr)\Longrightarrow
  \mathcal N(0,\Sigma_n).
\]
Here \(\Sigma_n\) is the asymptotic covariance matrix of the estimated
log-ratio vector under the fixed protocol.  Consequently, for the reconstructed
total, there is a finite constant \(V_n\ge0\) such that
\[
  \sqrt B\bigl(\log\wideM(n)-\log M(n)\bigr)
  \Longrightarrow
  \mathcal N(0,V_n).
\]
Consequently,
\[
  \SE(\log\wideM(n))
  =
  C_n B^{-1/2}+o(B^{-1/2}),
  \qquad C_n=\sqrt{V_n}.
\]
\end{proposition}

\begin{proof}
For fixed \(n\), each sampled layer chain is a finite-state irreducible
aperiodic Markov chain with uniform stationary measure.  Because \(A\) and
\(R\) are bounded on the finite state space, the Markov-chain CLT applies to
their time averages.  With layer budgets growing in fixed positive proportions,
independent chains and seeds give a joint CLT for the pooled addable/removable
averages.  Because the exact means \(\E_k A\) and \(\E_{k+1}R\) in
the ratios are strictly positive, the delta method propagates this CLT first
to the adjacent log-ratios and then through the smooth finite-dimensional
reconstruction map \(\Phi:y\mapsto \log\wideM(n)\).
\end{proof}

\paragraph{Numerical protocol.}
\label{sec:numerical-protocol}
All numerical runs use a fixed uniform allocation across sampled layers.
Before a run, we fix the sampled layers, use of duality, number of chains,
recorded states per chain, burn-in, thinning, and seed range.  Each seed
repeats the same procedure and produces one complete reconstruction.  Here
``uniform'' refers to the allocation of effort over sampled layers; within each
layer the Markov chain targets the uniform measure on \(\Om_{n,k}\).
The production-level choices for \(M(10)\), including the sampled layer set,
chain layout, burn-in, thinning, reconstruction rule, and seed-level
uncertainty calculation, are specified in
Appendix~\ref{sec:m10-production-details}.

All reported uncertainties are computed at the seed level.  Internally, a seed produces
\[
  L_s(n)=\log\wideM_s(n)
\]
under the fixed protocol.  For known backtests we report
\[
  e_s(n)
  =
  L_s(n)-\log M(n)
  =\log(\wideM_s(n)/M(n)).
\]
For unknown cases, the same seed-level log estimates give seed standard
errors, percentile bootstrap intervals, jackknife standard errors, and
split-half diagnostics \cite{EfronTibshirani1993}.  These quantify variation
among independent repetitions of the fixed estimator; possible shared
finite-chain bias is probed separately by the burn-in/thinning,
chain-layout, and seed-level mixing checks in
Appendix~\ref{sec:numerical-validation-details}.  Numerical tables and figures report
log errors and standard errors in \(\log_{10}\) units; the theoretical
statements above use natural logarithms.

\section{Numerical Results and Reconstruction Validation}
\label{sec:results}

\subsection{Convergence, Scaling, and Backtests}
\label{sec:convergence-backtests}

The numerical experiments assess the reconstruction of the global count from
layerwise estimates of the mean numbers of addable and removable elements.  In \(\log_{10}\) units, the fixed uniform
estimator follows the expected Monte Carlo error law:
\[
  \SE(\log_{10}\wideM(n))\approx C_n^{(10)} B_{\rm tot}^{-1/2}
\]
where \(B_{\rm tot}\) is the total number of recorded post-burn-in Markov-chain
states, summed over all sampled layers, chains, and pooled seeds.  \cref{fig:m9-convergence} tests this behavior for the
\(M(9)\) production run by pooling increasing seed prefixes and recomputing the
full reconstruction.

\begin{figure}[t]
\centering
\includegraphics[width=0.88\textwidth]{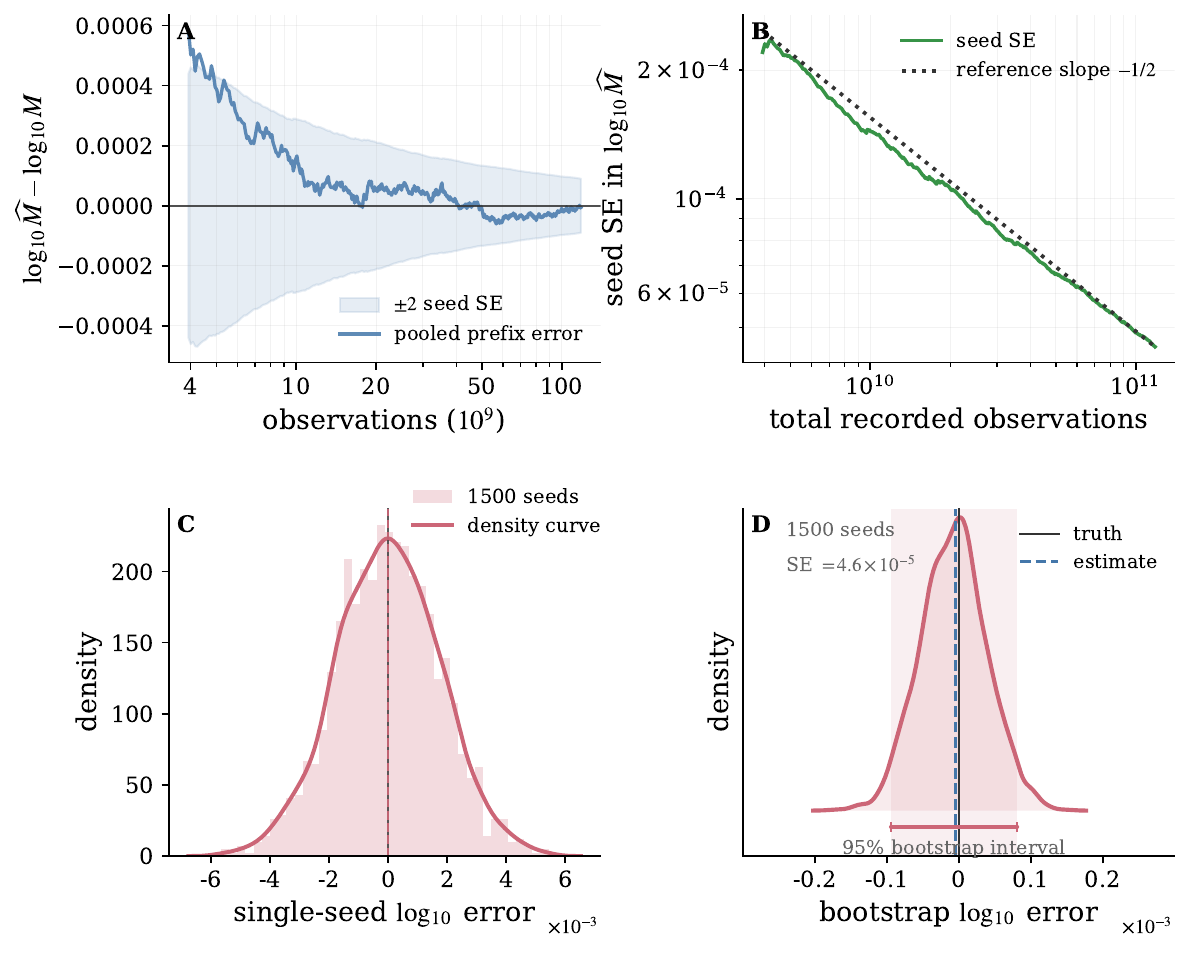}
\caption{\(M(9)\) convergence under seed-prefix pooling for the fixed
production protocol.  Panel A shows the pooled prefix \(\log_{10}\) error with
a \(\pm2\) seed-standard-error band.  Panel B shows the corresponding
\(\log_{10}\) seed standard error against the Monte Carlo reference slope
\(-1/2\).  Panel C shows the distribution of the 1500 single-seed
\(\log_{10}\) errors.  Panel D shows the seed-bootstrap distribution of the
final 1500-seed estimator.}
\label{fig:m9-convergence}
\end{figure}

The cross-\(n\) experiment estimates the dimensional growth of
\(C_n^{(10)}\).  For \(n=6,7,8,9\), we fit
\[
  \log\SE(\log_{10}\wideM(n);B_{\rm tot})=\alpha+\beta n
  -\frac12\log B_{\rm tot}+\varepsilon_{n,B},
\]
with the exponent fixed at the Monte Carlo value \(-1/2\).  The fitted values
are
\[
  \beta=0.620093,
  \qquad
  \exp(\beta)=1.8591,
  \qquad
  \exp(2\beta)=3.4563 .
\]
Thus, at a fixed budget, the seed-level \(\log_{10}\) standard error increases
by a factor of about \(1.86\) per added dimension, whereas maintaining a fixed
standard error requires about \(3.46\) times as many recorded states per added
dimension.  The fitted constant in
\(\log_{10}\) units for
\(n=10\) is \(C_{10}^{(10)}=28.2991\), so an \(M(10)\) run with total
recorded budget \(B_{\rm tot}\) has the seed-level forecast
\[
  \SE(\log_{10}\wideM(10))\approx 28.2991\,B_{\rm tot}^{-1/2}.
\]
This forecast sets the budget scale for \cref{sec:m10-protocol};
\cref{fig:cross-n-scaling} shows the fitted curves.

\begin{figure}[t]
\centering
\includegraphics[width=0.88\textwidth]{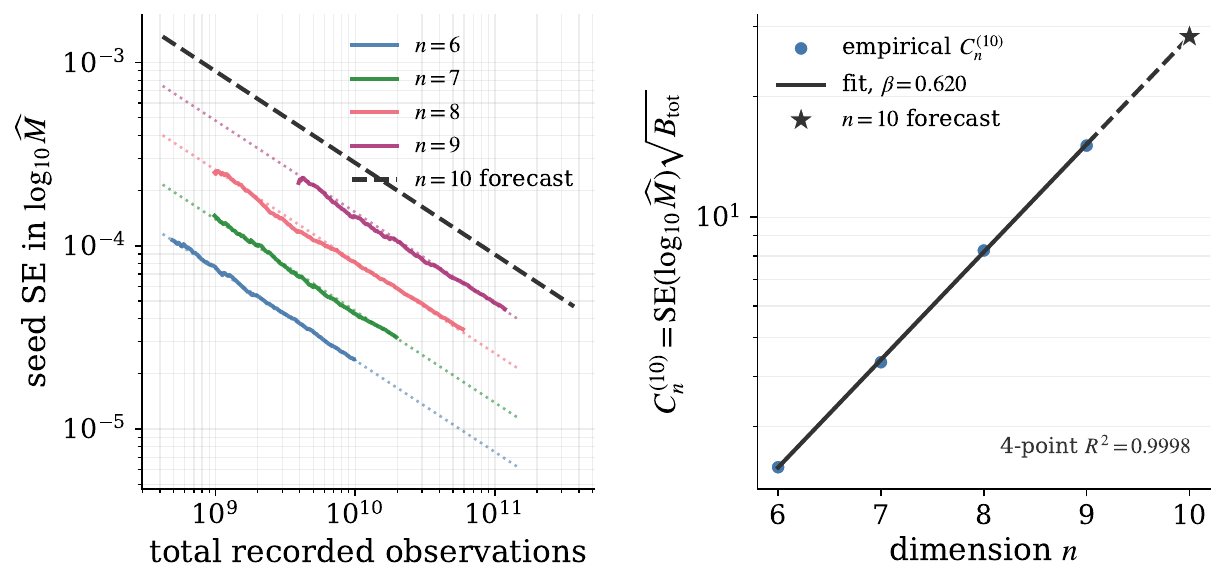}
\caption{Cross-\(n\) scaling of the seed-level \(\log_{10}\) variability for
\(n=6,7,8,9\), together with the fitted \(n=10\) forecast.  The fit constrains
the budget exponent to the Monte Carlo value \(-1/2\) and estimates the
remaining growth with \(n\).}
\label{fig:cross-n-scaling}
\end{figure}

The high-precision known-value runs give the finite-\(n\) benchmark points in
\cref{tab:high-precision-backtests}.  Their \(\log_{10}\) errors lie within the independently estimated
seed-level variability.  The full Monte Carlo
parameters are collected in \cref{tab:mc-parameters}.
For context, direct finite-\(n\) substitution of published asymptotic formulae
is reported in \cref{tab:asymptotic-comparison}.

\begin{table}[t]
\centering
\caption{High-precision known-value backtests.  Chains/layer is the actual
number of Markov chains run on each sampled layer.  Errors and standard errors
are in \(\log_{10}\) scale.}
\label{tab:high-precision-backtests}
\scriptsize
\setlength{\tabcolsep}{3.2pt}
\begin{tabular}{@{}rrrrrrr@{}}
\toprule
\(n\) & Seeds & Chains/layer
& \(\log_{10}\wideM(n)\) & \(\log_{10}\) err. & \(\log_{10}\) SE & \(z\) \\
\midrule
8 & 3000 & 2048 & \(22.749197804918\)
& \(-6.2026\times10^{-7}\) & \(3.5028\times10^{-5}\) & \(-0.02\) \\
9 & 1500 & 4096 & \(41.456947495132\)
& \(-5.1646\times10^{-6}\) & \(4.5119\times10^{-5}\) & \(-0.11\) \\
\bottomrule
\end{tabular}
\end{table}

\newpage
\subsection[Protocol estimate for M(10)]{Protocol estimate for \(M(10)\)}
\label{sec:m10-protocol}

The same estimator under the stated protocol can be applied to \(M(10)\), where
no exact value is available.  The known cases \(n\le9\) calibrate the budget
scale through the cross-\(n\) scaling experiment;
Appendix~\ref{sec:m10-production-details} records the layer set, chain layout,
burn-in, thinning, and reconstruction rule.  The resulting protocol estimate
is reported in \cref{tab:m10-production-result}.

\begin{table}[t]
\centering
\caption{\(M(10)\) protocol estimate and variability summary.}
\label{tab:m10-production-result}
\small
\begin{tabular}{@{}p{0.33\textwidth}p{0.20\textwidth}p{0.36\textwidth}@{}}
\toprule
Quantity & Value & Notes \\
\midrule
Total recorded states & \(314{,}572{,}800{,}000\) & production run total \\
Number of seeds & \(1000\) & independent production seeds \\
\(\log_{10}\wideM(10)\) & \(78.951142528342\) & primary estimate \\
\(\wideM(10)\) & \(8.9360\times 10^{78}\) & value-scale rendering \\
Seed SE for \(\log_{10}\wideM(10)\) & \(5.3089\times 10^{-5}\) & seed-level variability \\
Bootstrap interval & \(78.951143\pm 1.04\times10^{-4}\) & \(95\%\) \(\log_{10}\)-scale interval \\
Jackknife SE & \(5.3089\times 10^{-5}\) & stability check \\
Split-half difference & \(6.7109\times 10^{-5}\) & absolute \(\log_{10}\)-scale difference \\
Cross-\(n\) budget forecast & \(78.951143\pm5.0456\times10^{-5}\) & forecast from the fitted scaling law \\
\bottomrule
\end{tabular}
\end{table}

The seed and jackknife standard errors and the bootstrap interval in
\cref{tab:m10-production-result} quantify uncertainty arising from
repeat-to-repeat variability under the fixed protocol; they do not account for
systematic effects shared across seeds.  The bootstrap
half-width \(1.04\times10^{-4}\) on the \(\log_{10}\) scale corresponds to a
multiplicative half-width of about \(2.4\times10^{-4}\).  Separately, applying
the fitted cross-\(n\) scaling law
\(\SE(\log_{10}\wideM(10))\approx28.2991\,B_{\rm tot}^{-1/2}\) to the reported
production budget gives
\(\SE(\log_{10}\wideM(10))\approx5.0456\times10^{-5}\).  As a computational
scale reference, a 1000-seed run at about 30 minutes per seed on one RTX 5080
corresponds to roughly 500 RTX 5080 GPU-hours.  The cross-\(n\) quantity is an extrapolative budget
forecast based on lower-dimensional runs.

\subsection{Whitney-Number Shape}
\label{sec:layer-density-shape}

The estimator reconstructs the Whitney numbers, not only the total sum
\(M(n)\).  Here
\[
  a_n(k)=\#\{D\subseteq\Bn:\ D\text{ is a downset and }|D|=k\}
\]
is the \(k\)-th Whitney number of the ideal lattice \(\mathcal I(\Bn)\),
ranked by ideal cardinality.
With the two endpoint ideals omitted, the same rows appear in OEIS A269699 and
are described there as empirically unimodal \cite{OEISA269699}.  The shape of
this Whitney-number sequence is also connected to open Sperner-theoretic questions for
\(\mathcal I(\Bn)\), including unimodality,
RUSS, Peck-type consequences, and symmetric chain decompositions
\cite{McHard2009,Engel1997,GreeneKleitman1976,ProctorSaksSturtevant1980}.

\begin{figure}[t]
\centering
\includegraphics[width=0.88\textwidth]{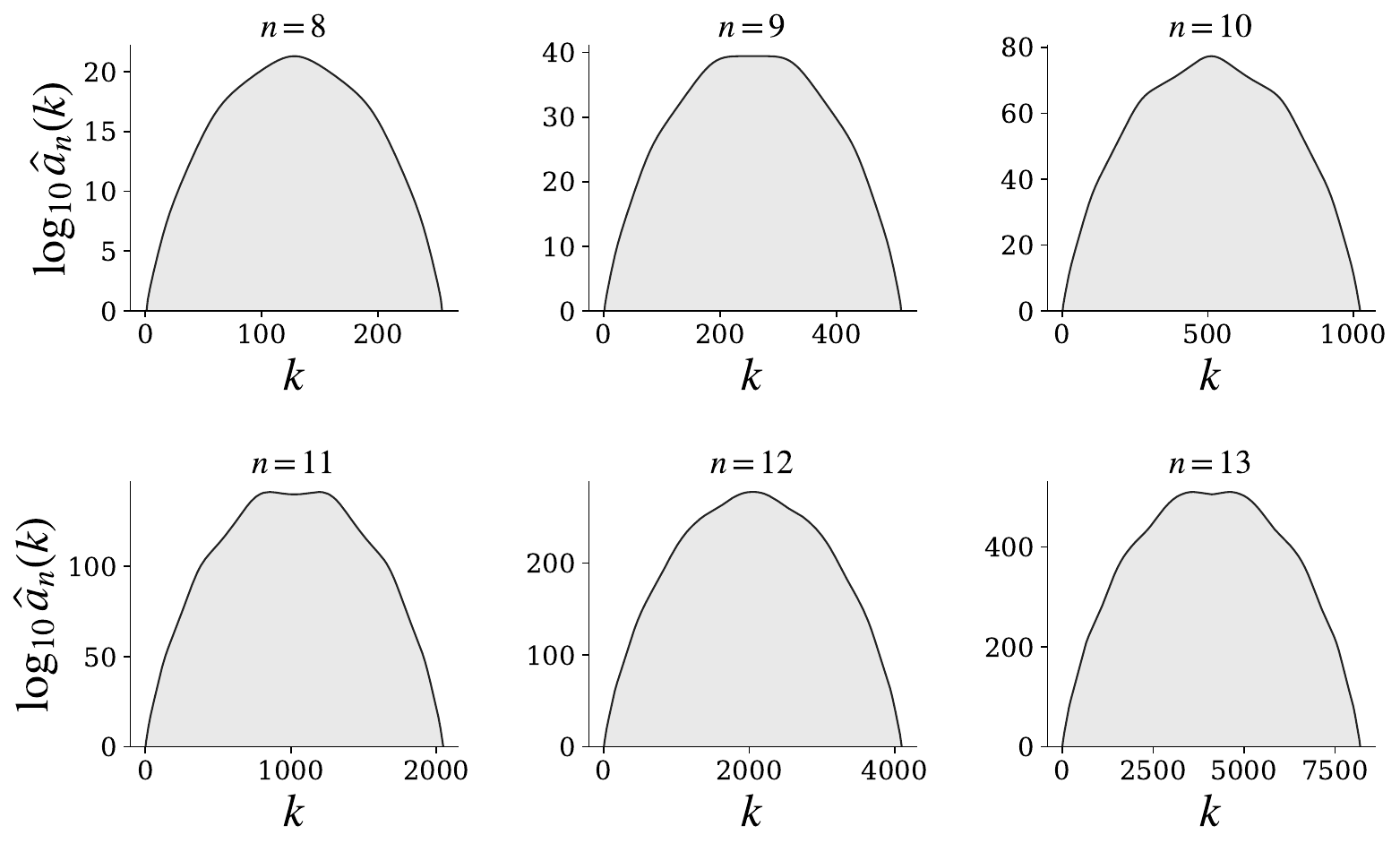}
\caption{Estimated Whitney-number sequences for \(n=8,\ldots,13\), displayed as
\(\log_{10}\widea_n(k)\).  The \(n=8\) and
\(n=9\) panels are high-precision backtest estimates, \(n=10\) is the
protocol estimate, and \(n=11,12,13\) are center-window estimates.}
\label{fig:layer-density-n8-to-n13}
\end{figure}

\cref{fig:layer-density-n8-to-n13} shows the estimated Whitney-number sequences for
\(n=8,\ldots,13\).  The odd cases display a center valley between two
symmetric shoulders.

\begin{figure}[t]
\centering
\includegraphics[width=0.88\textwidth]{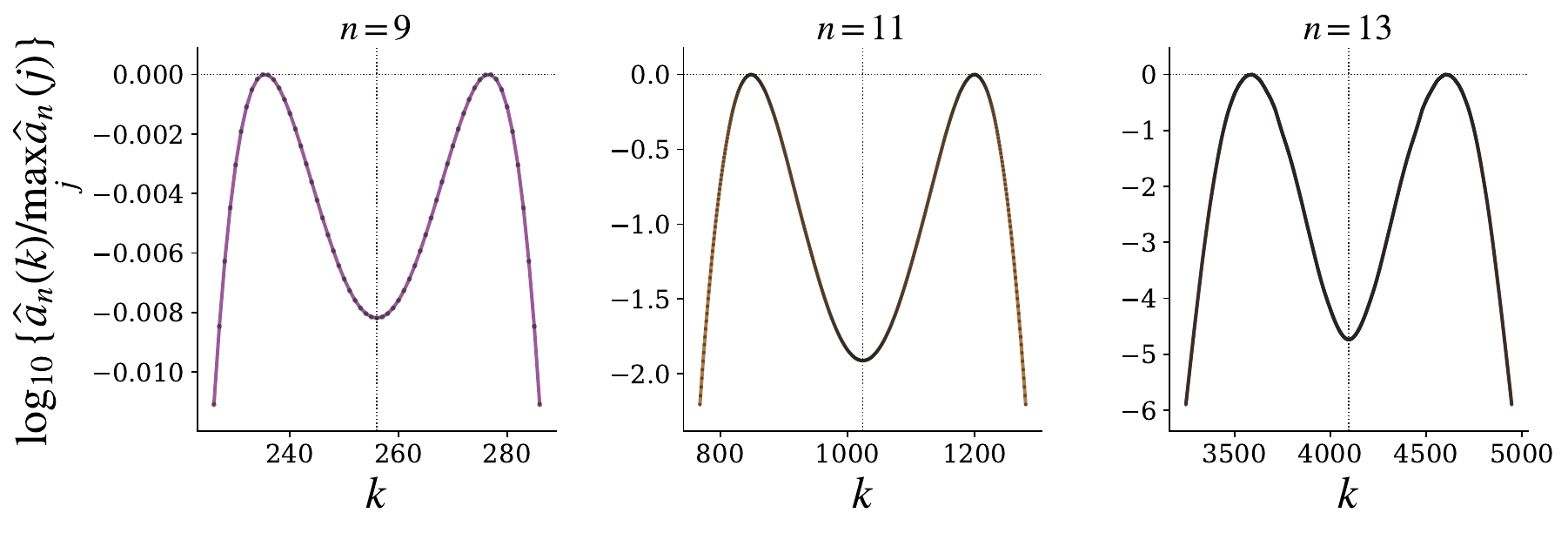}
\caption{Center zooms for the odd cases \(n=9,11,13\), displayed as
\(\log_{10}\{\widea_n(k)/\max_j\widea_n(j)\}\) within each plotted window.
The \(n=9\) panel comes from the high-precision
estimate and shows a shallow but resolved double shoulder:
the sampled shoulder occurs at \(k=235\), its dual is \(k=277\), and the
shoulder-to-center ratio is about \(1.019\).  The \(n=11\) and \(n=13\) panels
are center-window estimates and show an analogous center-valley
pattern with larger contrast.}
\label{fig:odd-center-bimodality}
\end{figure}

For \(n=9\), the feature is small but numerically separated in the seed-level
contrast analysis.  The estimated Whitney-number sequence has symmetric shoulders at
\(k=235\) and \(k=277\).  The independent sampled contrast is
\[
  \log_{10}\widea_9(235)-\log_{10}\widea_9(256)
  =
  0.008184090729,
\]
with seed SE \(9.31\times10^{-6}\), bootstrap 95\% interval
\([0.008166,0.008202]\), and shoulder-to-center ratio \(1.019023\).
The same high-precision run has \(\log_{10}\) error \(-5.1646\times10^{-6}\) for
\(M(9)\).  The estimated \(n=9\) profile therefore
has two symmetric shoulders above the center and is inconsistent with the
empirical unimodality description in OEIS A269699.  Because RUSS, Peck, and symmetric-chain
properties imply rank-unimodality, an exact confirmation of this center valley
would also rule out those stronger rank-unimodal patterns for
\(\mathcal I(\Bn)\).

The \(n=11\) and \(n=13\) panels in \cref{fig:odd-center-bimodality} show
the same picture in the plotted windows.  The shoulder-to-center ratios in
the estimated windows are approximately \(81.6\) for \(n=11\) and
\(5.46\times 10^4\) for \(n=13\).

\section{Discussion}

The method provides a finite-\(n\) Monte Carlo reconstruction that complements
exact enumeration.  It recovers the Whitney-number profile from local
adjacent-layer ratios and sums that profile to estimate \(M(n)\).  Asymptotic and cluster-expansion formulae remain
important reference points for the scale of \(M(n)\), but they do not by
themselves provide a data-driven uncertainty estimate for a specific unknown
finite case.

Although the reconstruction is a product of adjacent ratios on the original
scale, the calculation is carried out in log space.  Local ratio errors
therefore accumulate additively in the reconstructed log Whitney numbers.  To
first order,
\[
  \operatorname{Var}(\log\wideM(n))\approx s^\top\Sigma s,
\]
where \(s\) is the sensitivity of the log-sum-exp reconstruction to the
adjacent log-ratios and \(\Sigma\) is their Monte Carlo covariance.  The
known-value tests and cross-\(n\) scaling measure this accumulated error after
reconstruction.

The two-shoulder structure  in odd dimensional cases is a counterintuitive feature of the reconstructed
Whitney-number profile. Rank-unimodality of this sequence is an empirical
description in this setting, not a consequence of the exact rank duality. The
duality $a_n(k)=a_n(2^n-k)$ only enforces symmetry of the profile about the
center. It does not require the central coefficient $a_n(2^{n-1})$ to be
maximal. Therefore a symmetric two-shoulder profile, with a center valley
between two equal off-center shoulders, is compatible with the duality even
though it contradicts unimodality. A mathematical explanation of this structure remains
open; the evidence presented here is numerical and restricted to finite
\(n\).

The observed cross-dimensional scaling concerns the prefactor in the fixed-\(n\)
Monte Carlo error law.  For each fixed \(n\), the finite-state Markov-chain CLT
and the delta method give
\[
  \SE(\log \wideM(n))=C_n B^{-1/2}+o(B^{-1/2}),
\]
where \(B\) is the recorded-state budget.  The cross-\(n\) fit estimates how
the constant \(C_n\) changes with dimension.  In the tested dimensions this
change appears regular: at fixed recorded-state budget, the seed-level log
standard error grows by about a factor \(1.86\) per added dimension.  We use this
empirical regularity as a budget forecast for \(C_n\).  One possible
explanation is that the final log-sum-exp is most sensitive to the part of the
reconstructed Whitney-number profile carrying the largest mass, so the
effective error propagation may be governed by a relatively narrow central
range of adjacent ratios.  If the local addable/removable statistics and chain
autocorrelations in that range change regularly with \(n\), then the fitted
prefactor \(C_n\) can also vary regularly.  This explanation is heuristic.

The known-value backtests at \(M(8)\) and \(M(9)\) are consistent with the
measured seed-level variability, and the cross-\(n\) experiment gives a
practical budget scale for \(M(10)\).  The principal limitation is that the
combinatorial identity determines the exact target ratios but does not
guarantee the finite-run accuracy of their Monte Carlo estimates.  Finite-run mixing, autocorrelation, and
error propagation through the reconstructed log Whitney numbers therefore
remain empirical diagnostics.

\section{Conclusion}

This work reformulates Dedekind-number estimation as the reconstruction of the
Whitney numbers of \(\mathcal I(\Bn)\) across fixed-cardinality downset layers.
Rather than sampling the total set of monotone Boolean functions directly, the
method estimates local averages of the numbers of addable and removable
elements on each layer.  The adjacent-layer double-counting identity converts
these averages into the ratios \(a_n(k+1)/a_n(k)\), and log-space accumulation
together with Boolean-lattice duality gives the full reconstructed profile
\((\widea_n(k))_{k=0}^{2^n}\).  The Dedekind-number estimate \(\wideM(n)\) is
then obtained by summing this profile.

Under the fixed protocol, the \(n=8\) and \(n=9\) estimates agree with the known
Dedekind numbers within the measured seed-level variability and exhibit the
expected Monte Carlo scaling.  For
\(n=10\), the same protocol gives
\[
  \wideM(10)=(8.9360\pm0.0010)\times10^{78},
\]
where the displayed uncertainty is the value-scale expression of the
cross-\(n\) budget forecast.  The reconstruction also supplies rank-shape
information.  For \(n=9\), the Whitney-number profile has two symmetric
shoulders above the center, contrary to the empirical unimodality description
in OEIS A269699.  The \(n=11\) and \(n=13\) center-window estimates show the
same center-valley pattern with larger contrast.

\section*{Data Availability}

Code for reproducing the reported reconstructions are available at \url{https://github.com/mitotic0124/DedekindLayerMC}.

\section*{Acknowledgments}

The authors thank Ruiqing Xia, Yun Zhu, Shang Xiang, and Lan-Xi Tang for
helpful discussions and feedback.

\bibliographystyle{alphaurl}
\bibliography{references}

\appendix
\label{appendix}
\section{Protocol and Numerical Validation Details}
\label{sec:numerical-validation-details}

This appendix documents the \(M(10)\) production protocol and the numerical
tests underlying the empirical validation reported in the main text.

\subsection{\(M(10)\) Production Protocol}
\label{sec:m10-production-details}

This subsection records the fixed statistical protocol used for the reported
\(M(10)\) estimate.  It specifies the sampled layers, chain layout, reconstruction
map, and seed-level summaries; it is not meant to describe file formats,
paths, or other engineering details of the computation.

\begin{protocol}[\(M(10)\) production protocol]
\label{prot:m10-appendix}
Set \(n=10\) and \(N=2^{10}=1024\).  One production seed means one independent
repetition of the following fixed protocol.

\begin{enumerate}
  \item \textbf{Sampled layers and duality.}
        The sampled half-row consists of the layers
        \(k=1,\dots,512\).  The endpoint layer \(k=0\) is exact, and the
        remaining layers \(k=513,\dots,1024\) are filled by the exact Boolean
        duality
        \[
          a_{10}(k)=a_{10}(1024-k).
        \]
        Thus the reconstruction uses sampled information only on one side of
        each dual pair, together with the exact endpoint value \(a_{10}(0)=1\).

  \item \textbf{Fixed-layer chain layout.}
        For each sampled layer, the protocol runs
        \(4\times2048=8192\) fixed-layer chains.  Each chain records \(75\)
        post-burn-in states, using burn-in \(2500\) and thinning \(40\).
        Each chain is initialized by starting from the empty downset and adding
        uniformly chosen addable vertices until the target layer \(k\) is
        reached.

  \item \textbf{Fixed-layer transition.}
        Within a sampled layer, the Markov step is the exchange proposal from
        \cref{sec:fixed-layer-sampling}: delete a uniformly chosen removable
        vertex, add a uniformly chosen addable vertex, and accept the proposed
        state \(\Gamma\) from the current state \(D\) with
        \[
          \alpha(D,\Gamma)=\min\{1,R(D)/R(\Gamma)\}.
        \]
        Independent seeds use a fixed deterministic seed convention and
        independent random streams.

  \item \textbf{Recorded addable/removable averages.}
        For each recorded state \(D\), compute the addable count \(A(D)\) and
        removable count \(R(D)\).  For a fixed seed \(s\), the chains on layer
        \(k\) are pooled to form
        \[
          \wideA_{k,s}
          =
          \frac{1}{S_{k,s}}
          \sum_{c,t} A(D_{k,c,t}^{(s)}),
          \qquad
          \wideR_{k,s}
          =
          \frac{1}{S_{k,s}}
          \sum_{c,t} R(D_{k,c,t}^{(s)}),
        \]
        where \(D_{k,c,t}^{(s)}\in\Omega_{10,k}\) is the \(t\)-th recorded state
        of chain \(c\) in seed \(s\), and \(S_{k,s}\) is the number of recorded
        states pooled on that layer.  The implementation also records second
        moments, acceptance/change rates, and block summaries for the validation
        checks in \ref{sec:appendix-validation-checks}.

  \item \textbf{Seed-level reconstruction.}
        For each seed \(s\), set the exact endpoint mean
        \(\wideA_{0,s}=1\).  Form the adjacent log-ratios on the sampled side:
        \[
          \widehat y_{k,s}
          =
          \log \wideA_{k,s}-\log \wideR_{k+1,s},
          \qquad 0\le k<512 .
        \]
        Accumulate log Whitney numbers from the endpoint:
        \[
          x_{0,s}=0,
          \qquad
          x_{k,s}=\sum_{j=0}^{k-1}\widehat y_{j,s}
          \quad (1\le k\le512).
        \]
        Complete the full row by exact rank duality:
        \[
          \widehat x_{k,s}
          =
          \begin{cases}
            x_{k,s}, & 0\le k\le512,\\
            x_{1024-k,s}, & 512<k\le1024.
          \end{cases}
        \]
        The seed-level reconstructed Whitney numbers are
        \[
          \widea_{10,s}(k)=\exp(\widehat x_{k,s}).
        \]
        The seed-level Dedekind-number estimate is computed by log-sum-exp:
        \[
          L_s
          =
          \log\wideM_s(10)
          =
          m_s+\log\sum_{k=0}^{1024}\exp(\widehat x_{k,s}-m_s),
          \qquad
          m_s=\max_{0\le k\le1024}\widehat x_{k,s}.
        \]
        No fitted smoothing weights or penalty parameters are introduced in this
        reconstruction.

  \item \textbf{Seed-level summary and uncertainty calculation.}
        The reported primary estimate is obtained from the seed-level log
        estimates \(L_s=\log\wideM_s(10)\), weighting each seed by its recorded
        state count.  The seed standard error is the corresponding weighted
        mean standard error.  The bootstrap resamples seeds with replacement
        using the same weights; the jackknife leaves out one seed at a time and
        recomputes the weighted mean; the split-half diagnostic randomly
        partitions the seed list into two halves with a fixed random seed and
        reports the difference of the two weighted means.
\end{enumerate}
\end{protocol}

\subsection{Numerical Validation Checks}
\label{sec:appendix-validation-checks}

This subsection collects the numerical parameter summary and validation checks
other than the \(M(10)\) production protocol: finite-\(n\) asymptotic
comparison, exact small-\(n\) enumeration, truth-free reconstruction closure,
burn-in/thinning drift, production-scale mixing diagnostics, and
chain-structure invariance.  When the diagnostic summaries are produced in
natural-log units, the table entries below are converted to \(\log_{10}\)
units.

\begin{center}
\refstepcounter{table}
\label{tab:mc-parameters}
\footnotesize
\textsc{Table \thetable.}\enspace
Monte Carlo parameter summary for the numerical experiments.  Here
\(c\) is the actual number of Markov chains run on each sampled layer, and
\(m\) is the number of recorded states per chain.
\par\medskip
\centering
\scriptsize
\renewcommand{\arraystretch}{1.25}
\setlength{\tabcolsep}{2.5pt}
\resizebox{\textwidth}{!}{%
\begin{tabular}{@{}ccccccc@{}}
\toprule
Experiment & Layers & \(c\) & \(m\) & burn-in & thin.
& Seeds \\
\midrule
\(M(8)\) high precision / convergence
& \(n=8\), duality, 128 sampled
& 2048 & 75 & 2500 & 40
& 3000 \\
\(M(9)\) high precision
& \(n=9\), duality, 256 sampled
& 4096 & 75 & 2500 & 40
& 1500 \\
\(M(10)\) production estimate
& \(n=10\), duality, 512 sampled
& 8192 & 75 & 2500 & 40
& 1000 \\
Cross-\(n\) scaling
& \(n=6,7,8,9\), duality
& see note & 75
& 2500 & 40
& \(1000,1000,3000,1500\) \\
Truth-free closure
& \(n=8\), all nontrivial layers
& 256 & 75 & 2500 & 40
& 20 \\
Burn-in/thinning drift
& \(n=8\), duality, 128 sampled
& 256 & 75
& \(2500,5000,10000\)
& \(40,80,120\)
& \(40,40,40\) \\
Chain-layout invariance
& \(n=8\), duality, 128 sampled
& see note & see note & 2500 & 40
& 20 per layout \\
\bottomrule
\end{tabular}
}
\renewcommand{\arraystretch}{1}
\par\smallskip
\footnotesize
For the cross-\(n\) scaling experiment, the chains per sampled layer for
\(n=6,7,8,9\) are \(4096,4096,2048,4096\), respectively; the budget variable is
the total recorded observations obtained by adding independent seeds.  The
chain-layout rows use three equal-budget layouts:
\(256\times75\), \(64\times300\), and \(16\times1200\)
chains by recorded states per chain.
\end{center}

\noindent\textbf{Finite-\(n\) asymptotic comparison.}\par\smallskip
The formula of Korshunov and Sapozhenko, in the normalization used by
Jenssen, Malekshahian and Park, is the \(j\le2\) truncation of their
cluster-expansion expression
\cite{Korshunov1977,Korshunov2003,JenssenMalekshahianPark2024}.  The same
paper gives the next polynomial coefficients \(P_3^r\) and \(P_4^r\), so one
can also form \(j\le3\) and \(j\le4\) cluster-expansion truncations.  Because these formulae are asymptotic rather than finite-\(n\) numerical
guarantees, \cref{tab:asymptotic-comparison} reports their direct finite-\(n\)
evaluations at the largest dimensions with known values.

\begin{center}
\refstepcounter{table}
\label{tab:asymptotic-comparison}
\footnotesize
\textsc{Table \thetable.}\enspace
Direct finite-\(n\) substitution of published asymptotic formulae,
with the corresponding \(\log_{10}\) errors from this work where a
high-precision known-value backtest is reported.  Errors are relative to the
exact \(M(n)\).
The \(j\le2\) row is the Korshunov--Sapozhenko approximation, equivalently the
first two cluster-expansion terms in the normalization of
\cite{JenssenMalekshahianPark2024}.
\par\medskip
\centering
\scriptsize
\setlength{\tabcolsep}{2.5pt}
\resizebox{\textwidth}{!}{%
\begin{tabular}{@{}rlrrrrr@{}}
\toprule
\(n\) & Formula & Exact
& Asymp.
& Asymp. err.
& This work
& This work err. \\
& & \(\log_{10}M(n)\)
& \(\log_{10}\widetilde M(n)\)
& \(\log_{10}(\widetilde M/M)\)
& \(\log_{10}\wideM(n)\)
& \(\log_{10}(\wideM/M)\) \\
\midrule
7 & KS / CE \(j\le2\) & \(12.382859952093\) & \(12.363048150434\) & \(-1.9812\times10^{-2}\) & -- & -- \\
7 & CE \(j\le3\)     &                       & \(12.488696845908\) & \(1.0584\times10^{-1}\) &    &    \\
7 & CE \(j\le4\)     &                       & \(12.576910945211\) & \(1.9405\times10^{-1}\) &    &    \\
8 & KS / CE \(j\le2\) & \(22.749198425175\) & \(22.687132301056\) & \(-6.2066\times10^{-2}\) & \(22.749197804918\) & \(-6.2026\times10^{-7}\) \\
8 & CE \(j\le3\)     &                       & \(22.717315202061\) & \(-3.1883\times10^{-2}\) &    &    \\
8 & CE \(j\le4\)     &                       & \(22.731045948010\) & \(-1.8152\times10^{-2}\) &    &    \\
9 & KS / CE \(j\le2\) & \(41.456952659651\) & \(40.947270539746\) & \(-5.0968\times10^{-1}\) & \(41.456947495132\) & \(-5.1645\times10^{-6}\) \\
9 & CE \(j\le3\)     &                       & \(41.115843278861\) & \(-3.4111\times10^{-1}\) &    &    \\
9 & CE \(j\le4\)     &                       & \(41.228984400789\) & \(-2.2797\times10^{-1}\) &    &    \\
\bottomrule
\end{tabular}
}
\par\smallskip
\footnotesize
Blank entries repeat the exact value or this-work estimate shown on the first
row for the same \(n\).
\end{center}

\noindent\textbf{Exact small-\(n\) enumeration.}\par\smallskip
As an independent implementation check before using MCMC, exact enumeration of
all downsets of \(\Bn\) for \(0\le n\le7\) verified the adjacent-edge
double-counting identity and deterministic reconstruction, with the known
complete Whitney-number rows through \(n=7\) used as an external reference
\cite{OEISA269699}.

\noindent\textbf{Truth-free reconstruction closure.}\par\smallskip
The full-layer \(n=8\) closure runs checked the endpoint identity
\[
  \sum_{k=0}^{N-1}\log \rho_k=0.
\]
The estimated closure residual
\[
  C=\sum_{k=0}^{N-1}\log\widehat\rho_k
\]
was statistically compatible with zero; the largest standardized mean
residual over the tested budgets was \(1.78\).

\noindent\textbf{Burn-in and thinning drift.}\par\smallskip
The burn-in/thinning comparison repeats the same fixed \(n=8\) protocol with
the same 40 seeds in each setting.  The paired differences in
\cref{tab:burnin-thinning-drift} show no monotone drift; the largest paired
standardized shift is \(1.68\).

\begin{table}[t]
\centering
\caption{Paired burn-in/thinning drift diagnostic at \(n=8\).  Entries compare
mean \(\log_{10}\) errors using the same 40 seeds in each setting.  Here
\((b,t)\) is the burn-in/thinning pair.}
\label{tab:burnin-thinning-drift}
\scriptsize
\begin{tabular}{lrrrr}
\toprule
Comparison & Mean diff. & SE & \(z\) & Max abs. diff. \\
\midrule
\((5000,80)-(2500,40)\)
& \(-1.7455\times10^{-3}\) & \(1.3538\times10^{-3}\)
& \(-1.29\) & \(1.7668\times10^{-2}\) \\
\((10000,120)-(2500,40)\)
& \(4.6248\times10^{-4}\) & \(1.0337\times10^{-3}\)
& \(0.45\) & \(1.9154\times10^{-2}\) \\
\((10000,120)-(5000,80)\)
& \(2.2080\times10^{-3}\) & \(1.3151\times10^{-3}\)
& \(1.68\) & \(1.9773\times10^{-2}\) \\
\bottomrule
\end{tabular}
\end{table}

\noindent\textbf{Seed-level mixing diagnostic at \(n=10\).}\par\smallskip
As a production-level mixing diagnostic, we reprocessed 100 independent
\(M(10)\) seeds selected from the diagnostic rerun.  For each seed and sampled
layer, the post-burn-in recorded states were split into five consecutive
blocks; within each block, chains were grouped into 16 chain groups.  We
compared block ranges after averaging over chain groups and chain-group ranges
within each block.  The summaries in \cref{tab:m10-mixing-diagnostic} show high
acceptance and change rates, with no systematic block drift or persistent
chain-group separation at the scale of the recorded addable/removable means.

\begin{table}[t]
\centering
\caption{\(M(10)\) seed-level mixing diagnostic over 100 seeds selected from
the diagnostic rerun.
Mean averages all recorded seed-layer-block-chain-group summaries.  Block
med. and Block 95\% summarize, over seed-layer pairs, the range across the
five post-burn-in block means after averaging over chain groups.  Group med.
and Group 95\% summarize, over seed-layer-block triples, the range across the
16 chain groups.}
\label{tab:m10-mixing-diagnostic}
\scriptsize
\begin{tabular}{lrrrrr}
\toprule
Metric & Mean & Block med. & Block 95\% & Group med. & Group 95\% \\
\midrule
Acceptance rate
& \(0.990863\) & \(8.14\times10^{-5}\) & \(2.00\times10^{-4}\)
& \(5.01\times10^{-4}\) & \(1.16\times10^{-3}\) \\
Change rate
& \(0.973761\) & \(1.40\times10^{-4}\) & \(3.04\times10^{-4}\)
& \(8.50\times10^{-4}\) & \(1.73\times10^{-3}\) \\
\(\overline A\)
& \(83.6393\) & \(6.10\times10^{-2}\) & \(1.48\times10^{-1}\)
& \(3.86\times10^{-1}\) & \(8.12\times10^{-1}\) \\
\(\overline R\)
& \(64.8802\) & \(4.06\times10^{-2}\) & \(1.04\times10^{-1}\)
& \(2.64\times10^{-1}\) & \(5.90\times10^{-1}\) \\
\bottomrule
\end{tabular}
\end{table}

\noindent\textbf{Chain-structure invariance.}\par\smallskip
At fixed per-layer budget, we compared the three chain layouts in
\cref{tab:chain-layout}.  The mean \(\log_{10}\) errors stay within
seed-level noise, and the largest pairwise \(z\)-score is about \(0.87\).

\begin{table}[t]
\centering
\caption{Chain-layout invariance at \(n=8\).  The per-layer recorded-state
budget is held fixed while the number and length of chains are changed.  Error
columns are in \(\log_{10}\) units.}
\label{tab:chain-layout}
\scriptsize
\setlength{\tabcolsep}{3pt}
\begin{tabular}{lrrrrrrr}
\toprule
Layout & Chains & States/chain & States/layer & Seeds
& Mean & SD & SE \\
\midrule
many short & 256 & 75   & 19200 & 20
& \(-0.0003220\) & \(0.005350\) & \(0.001196\) \\
balanced   & 64  & 300  & 19200 & 20
& \(-0.0006383\) & \(0.005968\) & \(0.001334\) \\
few long   & 16  & 1200 & 19200 & 20
& \(0.0008882\) & \(0.005056\) & \(0.001131\) \\
\bottomrule
\end{tabular}
\end{table}

\end{document}